\documentclass[11pt, letterpaper]{amsart}
\usepackage[left=1in,right=1in,bottom=1in,top=1in]{geometry}
\usepackage{graphicx}
\usepackage{amsmath,amsthm,amssymb}
\usepackage{mathrsfs}
\usepackage[all,cmtip]{xy}
\usepackage{setspace}
\usepackage{color}
\usepackage[usenames,dvipsnames,svgnames,table]{xcolor}
\usepackage{url}
\usepackage{multirow}
\usepackage{subcaption}
\usepackage[numbers]{natbib}

\theoremstyle{plain}

\newcommand{\bp}{[\#\mbox{bp}]}

\newcommand{\ENJ}{\mathcal{E}_{\mbox{\scriptsize NJ}}}

\newcommand{\Var}{\mbox{Var\,}}
\newtheorem{lem}{Lemma}
\newtheorem{thm}[lem]{Theorem}

\begin{document}

\title{Deconvolving RNA base pairing signals}

\author{Torin Greenwood        \and Christine E.~Heitsch %etc.
}

\begin{abstract}
A growing number of \textsc{rna} sequences are now known to have distributions of multiple stable sequences.  Recent algorithms use the list of nucleotides in a sequence and auxiliary experimental data to predict such distributions.  Although the algorithms are largely successful in identifying a distribution's constituent structures, it remains challenging to recover their relative weightings.  In this paper, we quantify this issue using a total variation distance.  Then, we prove under a Nussinov-Jacobson model that a large proportion of \textsc{rna} structure pairs cannot be jointly reconstructed with low total variation distance.  Finally, we characterize the uncertainty in predicting conformational ratios by analyzing the amount of information in the auxiliary data.

%\textsc{rna} secondary structure prediction algorithms were originally developed to determine the single structure of an \textsc{rna} sequence.  When guided by auxiliary experimental data, these methods have high enough accuracies to recover functional information for large classes of sequences.  However, many sequences are now known to fold into multiple stable conformations.  In response, recent advancements in pseudoenergy models, such as Rsample, refine prediction methods to identify multiple structures within a distribution.  But, it remains challenging to predict the conformational ratios of these constituent structures.  
%
%In this paper, we explain mathematically why these challenges persist.  First, we systematically analyze how well Rsample and other existing models recover conformational ratios by using simulated experimental data and a total variation distance.  We conclude that the models have the tendency to overweight some structures.  Then, we use a Nussinov energy model to illustrate how the reconstruction of conformational ratios depends on model parameterization.  Finally, we use tools from information theory to analyze how much information auxiliary data contains about conformational ratios.  We find that some conformational ratios are easier to predict than others.

%\keywords{RNA secondary structure \and thermodynamic optimization with auxiliary data \and Fisher Information}
%\subclass{92D20 \and 05A15 \and 62P10}
\end{abstract}

\maketitle

\section{Introduction}

\label{intro}

Determining the structural conformations of an \textsc{rna} sequence reveals functional information.  Identifying structures in the lab is difficult, so discrete optimization methods are used instead.  These methods use the list of nucleotides in a sequence to predict two dimensional approximations of structures, called secondary structures, that still encode functional information.  However, an increasing number of \textsc{rna} molecules are now known to fold into multiple stable structures, \cite{Leonard:2013, Spasic:2018}.  For example, approximately 20\% of eukaryotic \textsc{rna} folds into multiple structural conformations \emph{in vivo}, \cite{Lu:2016}.  Differing conformations can also yield multiple biological functions for the same \textsc{rna} sequence, as is the case for riboswitches, \cite{Antunes:2018}.  In response, single-structure prediction methods have been refined to find distributions of \textsc{rna} structures, as reviewed in \cite{Schroeder:2018}.

When searching for single conformations, the single-structure prediction algorithms use the Nearest Neighbor Thermodynamic Model (\textsc{nntm}), \cite{mathews-turner-06, turner-mathews-10}.  The \textsc{nntm} can be used to assign an energy to any structure by considering all of the components of the structure.  Then, one popular approach is to find the minimum free energy (\textsc{mfe}) structure by using dynamic programming algorithms (based on the Zuker algorithm, \cite{Zuker:1981}).  Because the \textsc{nntm} is a model, these \textsc{mfe} structures do not always have perfect accuracy.  Plausible alternatives can be generated by predicting suboptimal structures, \cite{zuker-89, wuchty-etal-99}, and computing base pairing probabilities under a Boltzmann partition function, \cite{mccaskill-90, hofacker-etal-94, mathews-04}.  The space of suboptimal structures can be explored further by converting energies into Boltzmann probabilities and sampling from the corresponding distribution, \cite{Ding:2003}.

More recently, laboratory experiments have been used to direct single-conformation predictions.  Folded sequences are exposed to a chemical, and the chemical reactivity to each nucleotide is measured, \cite{Deigan:2009}.  The reactivity is positively correlated to a nucleotide's unpairedness, but due to a variety of experimental reasons, the signal is noisy.  Several competing methods of converting this noisy data into pseudoenergies have been proposed (\cite{Deigan:2009, Zarringhalam:2012, Washietl:2012}, reviewed in \cite{Eddy:2014}).  The pseudoenergies are then added to the energies from the \textsc{nntm}, shifting predictions according to the data.  As described in \cite{Deigan:2009}, when auxiliary data is included in single-structure predictions, the accuracy of these methods is high enough to recover important structural information about wide classes of \textsc{rna} sequences.

When adapting single-structure prediction algorithms to sequences with multiple conformations, auxiliary experimental data can be used in conjunction with Boltzmann distributions to predict distributions of structures.  A successful prediction must identify a distribution's constituent structures and their relative conformational ratios.  Unfortunately, even when a sequence has multiple conformations, existing prediction models collapse distributions towards one structure when incorporating experimental data (\cite{Spasic:2018}, see examples below in Section \ref{Sec:Examples}).  Rsample is a recent enhancement to these data-directed prediction algorithms designed to identify multimodal conformations, \cite{Spasic:2018}.  Rsample is largely successful in identifying the structural modes in a distribution.  However, identifying the correct conformational ratios remains challenging (see Section \ref{Sec:Examples}).

In this paper, we will give evidence for why it is difficult to reconstruct the conformational ratios in a multimodal structural distribution.  We focus on bimodal structural distributions because they are the simplest form of multimodal distribution.  In Section \ref{Sec:Examples}, we motivate the problem by using existing prediction methods to analyze \textsc{rna} sequences with two known conformations.  Currently, there is not a large collection of experimental data available from multimodal distributions of structures.  Thus, in order to analyze the performance of prediction methods systematically, we use empirically-derived distributions of auxiliary data from \cite{Sukosd:2013}.  With these distributions, we generate simulated data for known structures.   However, for a given \textsc{rna} sequence, the conformational ratio between structures is likely to change depending on a variety of extrinsic factors, including temperature, pH, Mg$^{2+}$ concentration, interactions with other nucleic acids, and ligand or protein binding, \cite{Hobartner:2003}.   Thus, we mix our simulated data sequences in different proportions, allowing us to analyze the behavior of prediction models as the proportion of structures in a distribution varies.  These examples will prompt mathematical definitions useful in pinpointing the problems in reconstructing weightings.

Next, in Section \ref{Sec:Nussinov}, we analyze mathematically the challenges of reconstructing conformational ratios.  Two pivotal factors in using auxiliary data are the way the data is incorporated into predictions, and how much information about the conformational ratios is contained in the auxiliary data.  To separate the factors, we first consider how well prediction models perform when given noiseless auxiliary data.   In order to remove the noise from the auxiliary data, we consider a binary model of data that describes whether each nucleotide in a structure is paired or unpaired, without revealing the complementary nucleotide in a pairing.  Then, we investigate a Nussinov-Jacobson model of assigning energies to structures, \cite{Nussinov:1980}, which is amenable to combinatorial analysis but still captures critical features of the other models.  Despite the Nussinov-Jacobson model's simplicity, its analysis has advanced our understanding of secondary structure prediction in other contexts, \cite{Clote:2005, Clote:2007}.  In Theorem \ref{ProbabilisticLimit}, we find:\\

\fbox{
\parbox{0.9 \linewidth}{
Conclusion 1: As the length of $\textsc{rna}$ sequences increases, almost all pairs of structures have some conformational ratio that cannot reliably be reproduced by a Nussinov-Jacobson energy model.
}
}

\ \\

After having analyzed the energy models directly, we then turn towards analyzing the noisy data in Section \ref{Sec:Information}.   Because the conformational ratio between structures may change depending on experimental factors, the auxiliary data must contain information about the conformational ratio that is not in the underlying thermodynamics of the sequence.  Some of the challenges of interpreting \textsc{shape} reactivities in multimodal distributions are revealed in \cite{Vieweger:2018}.  Motivated by the ability of Rsample to identify structures within a distribution and the ``sample and select'' algorithm from \cite{Quarrier:2010}, we assume that the structures in a distribution are known, but that their conformational ratio $p$ is not.  We then use tools from information theory to quantify how well the ratio can be recovered from the auxiliary data.  We find the following:\\

\fbox{
\parbox{0.9 \linewidth}{
Conclusion 2:  Let $n$ be the number of positions where the nucleotides in two structural modes are not both paired nor both unpaired.  Any unbiased estimator of the conformational ratio $p$ will have a variance that decreases at most linearly in $n$.
}
}

\section{Background and related work} \label{Sec:Background}
Before proceeding, we give some background information on auxiliary data and prediction models.

\subsection{Nearest Neighbor Thermodynamic Model} \label{NNTM}
The \textsc{nntm} has hundreds of parameters, corresponding to all the possible features in an \textsc{rna} structure, \cite{mathews-turner-06, turner-mathews-10}.  Structures are then assigned energy scores by summing the subscores of their components.  Because we will be analyzing multimodal distributions, we will look at the suboptimal structures predicted by this model.   Given a fixed \textsc{rna} sequence, energies can be used to define a Boltzmann distribution of structures: a structure $S$ with energy $\mathcal{E}(S)$ is assigned the probability,
 \begin{equation} \label{BoltzmannProbability}
 \mathbb{P}(S) := \frac{e^{-\mathcal{E}(S)/RT}}{Z},
 \end{equation}
 where  $Z$ is a partition function over all possible structures for the sequence, $R = 0.001987$ kcal/(mol K) is the gas constant, and $T$ is the temperature, which we take to have default value $310$K.  Note that when an \textsc{rna} sequence is predicted to fold into a single structure, a prediction algorithm typically aims to find the structure $S$ which minimizes the energy, or equivalently, maximizes $\mathbb{P}(S)$.  Alternatively, to search the space of suboptimal structures, there are efficient algorithms which can sample structures from the distribution, \cite{Ding:2003}.
 
 \subsection{SHAPE data and pseudoenergies} \label{SHAPE}

Auxiliary data can be used to reweight predicted distributions of structures by adding additional \emph{pseudoenergy} terms to each structure.  The data comes from laboratory experiments where folded \textsc{rna} sequences are exposed to chemicals, and the reactivity of each nucleotide is measured.  The reactivities are called \textsc{shape} data, named after the method, \emph{selective 2'-hydroxyl acylation analyzed by primer extension}, and they are positively correlated with unpairedness.  Different chemicals can be used, and each can give a different structural signal, \cite{Rice:2014}.

Although \textsc{shape} data encodes some information about a nucleotide's pairedness, the signal is noisy due to a variety of experimental and biological factors.  There are several methods for converting this signal into pseudoenergies to shift the Boltzmann distribution, and \cite{Eddy:2014} compares some common \textsc{shape} methods for single-structure prediction.  Below, we use the models from Deigan et al., \cite{Deigan:2009}, Zarringhalam et al., \cite{Zarringhalam:2012}, and Washietl et al., \cite{Washietl:2012}, and we use the implementation in the ViennaRNA software package, \cite{Lorenz:2011}.
%The accuracy of data-directed single-conformation predictions is typically high enough to recover functional information, \cite{Deigan:2009}.

Recently, new data-directed approaches aim to probe or reconstruct multimodal distributions.  We focus on Rsample, \cite{Spasic:2018}, which uses a single \textsc{shape} sequence and an advanced pseudoenergy model.  Rsample succeeds in identifying the constituent structures within several known \textsc{rna} structural distributions.  Identifying the correct ratios between structures remains challenging, as we will see in Section \ref{Sec:Examples} below.  Other prediction models include \textsc{Reeffit}, \cite{Cordero:2015}, which uses additional mutate-and-map information to guide predictions.  Finally, Ensemble\textsc{rna}, \cite{Woods:2017}, probes the structural space for potential structural modes without being limited by the exponential decay of Boltzmann probabilities.  All three programs are described in \cite{Schroeder:2018}.

\subsection{Modeling SHAPE data}
Although there is evidence of the existence of multimodal structural distributions, there is not yet a large collection of data derived from known multimodal distributions.  Thus, we model \textsc{shape} data coming from a multimodal distribution with a linear interpolation of \textsc{shape} data from two structures, as described in Section \ref{Sec:Examples}, and in line with \cite{Spasic:2018}.  To model the \textsc{shape} data from a single known structure, we use the empirically-derived distributions developed in \cite{Sukosd:2013}.  Here, the authors used \textsc{shape} data collected in \cite{Deigan:2009} from a 16S and a 23S \emph{Escherichia coli} ribosomal sequence to fit distributions for edge-paired, center-paired, and unpaired nucleotides.  A nucleotide is center-paired if it and both its neighbors belong to the same helix, or if it is paired and adjacent to a single-nucleotide bulge.  Otherwise, a paired nucleotide is edge-paired.

In Section \ref{Sec:Examples}, we use these empirical distributions to simulate data to illustrate the potential problem with reconstructing conformational ratios.  Later, in Section \ref{Sec:Information}, we use the empirical distributions again to analyze how much information \textsc{shape} data contains about the ratios.  In contrast, in Section \ref{Sec:Nussinov}, we use a noiseless binary model of \textsc{shape} data so that we can analyze the prediction models without confounding factors.
 
\subsection{Total variation distance} \label{Sec:TVD}

We choose a method to quantify error when predicting conformational ratios, based on the total variation distance.  The total variation distance is a common distance measure for probability measures, defined as the maximum difference the two measures could assign to any event.  When comparing two Bernoulli distributions with parameters $p$ and $\hat p$, the total variation distance is simply $|p - \hat p|$.

To apply this distance to distributions of \textsc{rna} structures, consider predicting a target bimodal distribution comprised of ratio $p$ of structure $A$ and $(1 - p)$ of structure $B$.  To quantify how well a prediction method can reconstruct this distribution, we define $\hat p$ to be the probability that a structure in the prediction method's Boltzmann distribution is closer to structure $A$ than it is to structure $B$ (by comparing F-measures).  We estimate $\hat p$ by taking a sample of 1000 structures, and then we call the absolute difference between this estimate and $p$ the total variation distance.  Note that the reduction of the Boltzmann distribution to only two classes of structures (those closer to $A$ and those closer to $B$) causes this total variation distance to be less than had we allowed some structures to be classified as a mismatch.  Below, we use this definition to find lower bounds on the error of common prediction methods in predicting $p$.  This error measure can easily be extended to distributions with more than two modes.

\section{Examples} \label{Sec:Examples}
\begin{figure}
\begin{center}
\includegraphics[width=\textwidth]{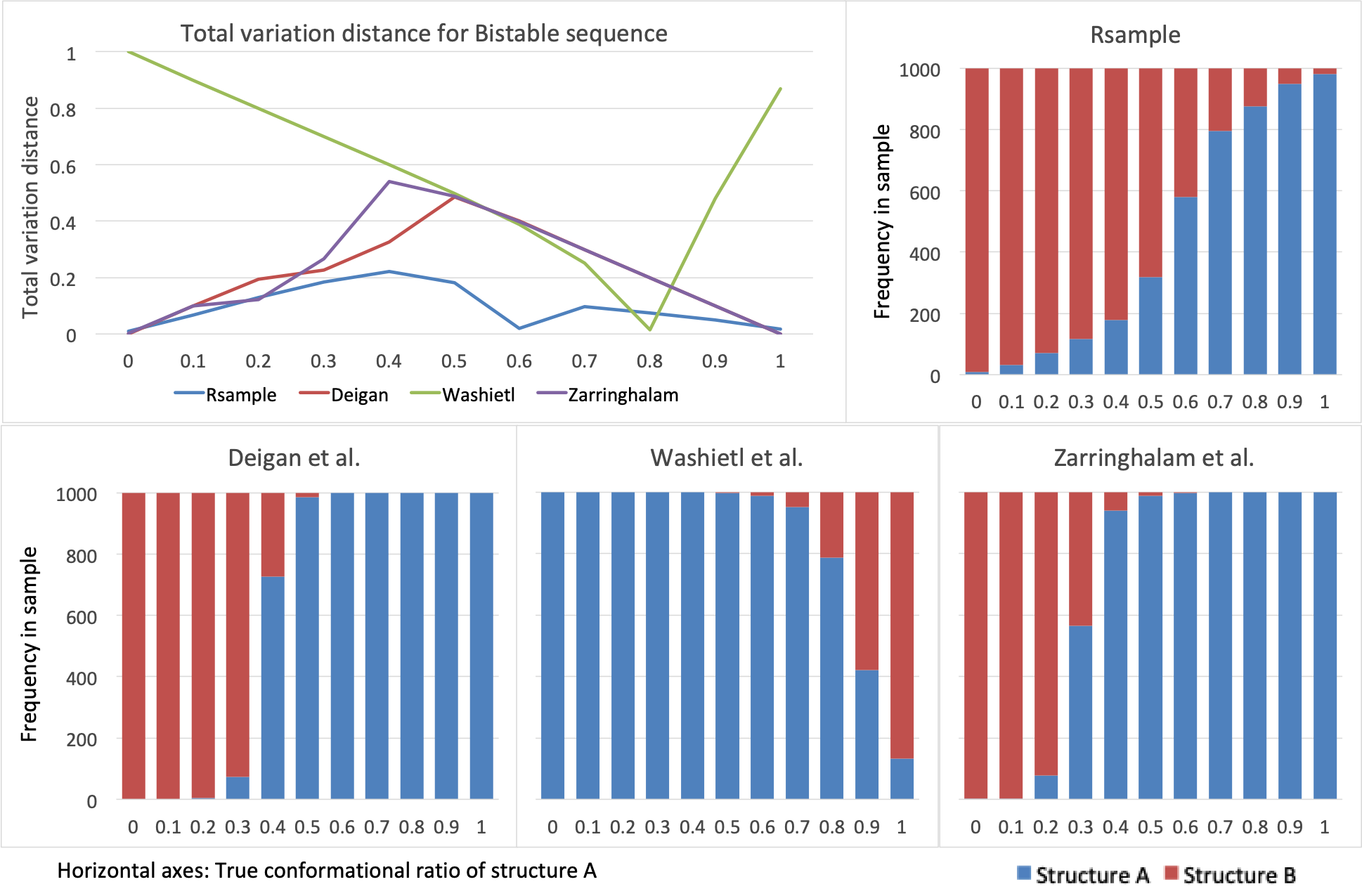}
\end{center}
\caption{Each prediction method generated samples of 1000 Bistable \textsc{rna} structures for each mixed auxiliary data sequence $M(p)$, with $p$ values ranging from $0$ to $1$.  The upper left graph illustrates the error in the conformational ratio, measured using a total variation distance.  
The bar graphs show the breakdown of samples into the two corresponding structures.
}
\label{fig:Bistable}
\hrule
\end{figure}
We begin with the synthetic 25-nucleotide Bistable 1 sequence from \cite{Hobartner:2003}, which is modeled after riboswitches.  In \cite{Spasic:2018}, Rsample was found to accurately reconstruct a bimodal Bistable distribution when given \textsc{shape} data corresponding to a known experimental setting.  However, the conformational ratio is expected to be dependent on external factors, so we will explore the accuracy of the models when the conformational ratio is changed.  To do so, we simulate auxiliary experimental data sequences $\{S_i\}$ and $\{T_i\}$ corresponding to the Bistable structures $A$ and $B$.  Then, we model mixtures of ratio $p$ of structure $A$ and $(1 - p)$ of structure $B$ with the linear interpolation, $M(p) = \{pS_i + (1 - p)T_i\}$.  We input the Bistable \textsc{rna} sequence and $M(p)$ into the prediction models from Deigan et al., \cite{Deigan:2009}, Zarringhalam et al., \cite{Zarringhalam:2012}, Washietl et al., \cite{Washietl:2012}, and Rsample, \cite{Spasic:2018}, for each value of $p$ between 0 and 1, incremented by 0.05.  We sampled 1000 structures from the Boltzmann distribution generated by each model.

Due to the simplicity of the Bistable sequence, over 92\% of the structures in every sample closely matched conformation $A$ or $B$ (with F-measure above $90\%$).  Thus, we instead measure how well the conformational ratio is predicted in Figure \ref{fig:Bistable} by using a total variation distance, as described in Section \ref{Sec:TVD}.  Rsample has the best performance overall, with an average total variation distance of under $0.1$.  Note that every method must get some ratio correct as long as the pseudoenergies they use are a continuous function of the auxiliary data.  Thus, even methods which perform poorly overall will have a small range of $p$ values where they are successful.  Also, the total variation distance is typically small for values of $p$ near $0$ and $1$, confirming that the \textsc{shape} data is able to direct single-conformation predictions.

To gain insight into why the total variation distance is sometimes large, we look at the breakdown of the samples into structures $A$ and $B$ in Figure \ref{fig:Bistable}.  For all but one of the methods, we see that as the weighting of structure $A$ increases in the auxiliary data, the predicted conformational ratio increases in the sample.  But, the total variation distance is high for the Deigan et al. and Zarringhalam et al. methods because structure $A$ dominates over structure $B$ for most $p$ values.

\begin{figure}
\begin{center}
\includegraphics[width=\textwidth]{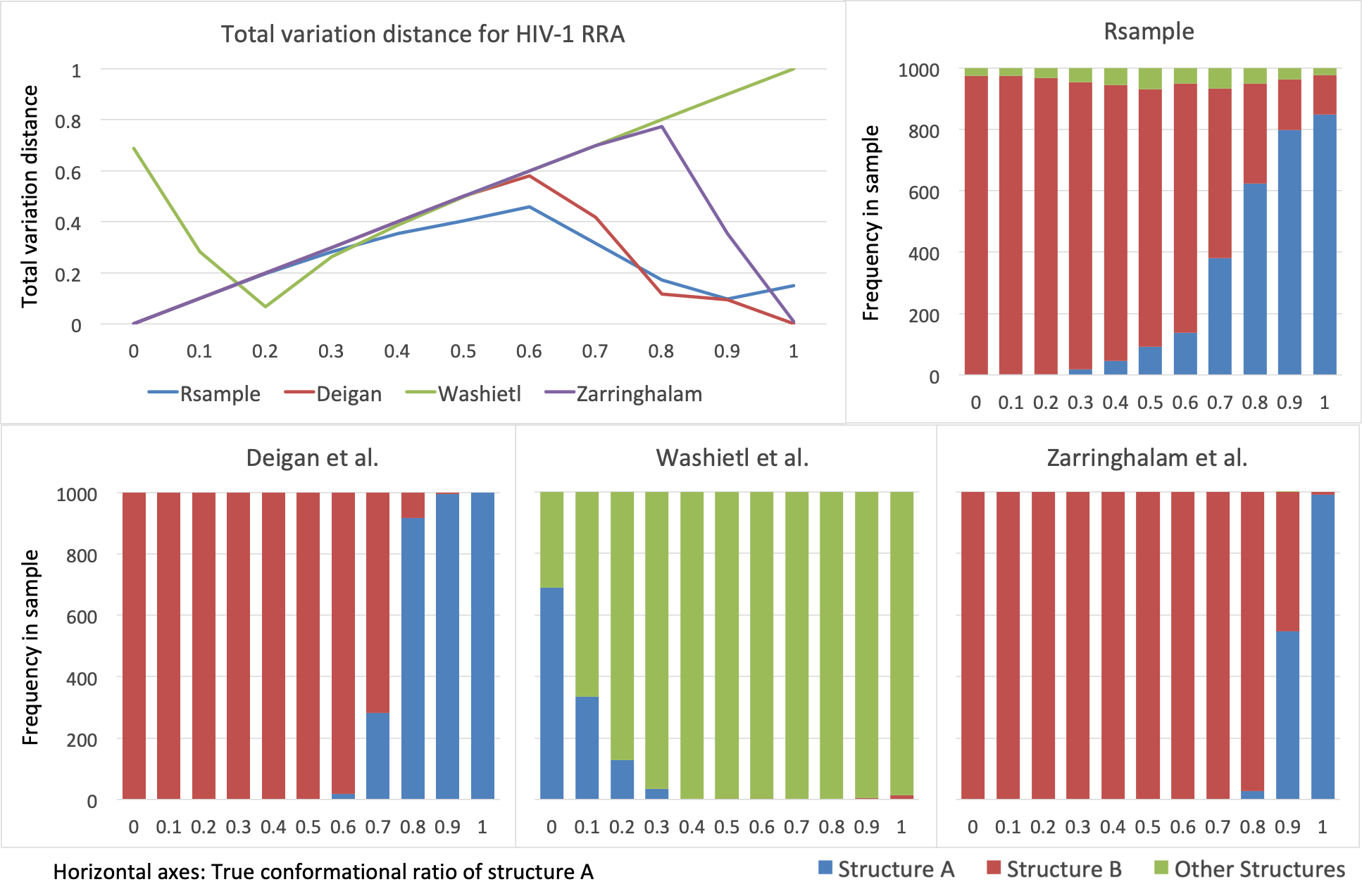}
\end{center}
\caption{We repeated the process from Figure \ref{fig:Bistable} for two \textsc{hiv-1 rra} structures. 
 }
\label{fig:HIV}
\hrule
\end{figure}

Unfortunately, predictions are worse for other sequences. In Figure \ref{fig:HIV}, we ran the same experiment with the \textsc{hiv-1 rra} sequence from the test set in the Rsample publication, \cite{Spasic:2018}.  Here, Rsample consistently overweights one of the two known \textsc{hiv-1 rra} conformations, as the other models did for the Bistable sequence.  However, Rsample still performs the best out of the methods we tested, with an average total variation distance of $0.23$.  Because the \textsc{hiv-1 rra} sequence is more complicated, some samples contained structures whose sets of helices did not match either known conformation, as shown in green in Figure \ref{fig:HIV}.

In this paper, we will show mathematically why it is difficult to use pseudoenergy models to reconstruct conformational ratios.  In Section \ref{Sec:Nussinov}, we will identify conditions in pseudoenergy predictions that can lead to the skewed samples in Figure \ref{fig:Bistable}, and in turn, high total variation distances.  Then, we show that under the Nussinov-Jacobson energy model and a simple random model for \textsc{rna} secondary structures, a large set of \textsc{rna} structures have conformational weightings that are hard to reconstruct.  Finally, in Section \ref{Sec:Information}, we use information theory to show how the ability to reconstruct a conformational ratio depends on the constituent structures in a distribution.

\section{Pseudoenergy models with noiseless auxiliary data} \label{Sec:Nussinov}

Here, we look at a simple energy model for \textsc{rna} structures to illustrate the challenges in reconstructing conformational weightings common to all energy models.  The Nussinov-Jacobson energy model, based on the algorithm from \cite{Nussinov:1980} and analyzed in \cite{Clote:2005, Clote:2007}, assigns energies to structures by totaling the number of base pairs in the structure. Thus, the structure with the most base pairs is considered optimal.  With this in mind, for any structure $A$, define $\bp(A)$ to be the number of paired nucleotides in $A$.  Then, we define an extension to the Nussinov-Jacobson model that incorporates auxiliary data: the energy of structure $A$ given auxiliary data sequence $M = \{M_i\}$ is:

\begin{equation} \label{Nussinov}
\ENJ(A | M) = -\bp(A) + \sum_{i = 1}^n C|x_i - M_i|,
\end{equation}
where $C > 0$ is a model parameter, and $x_i$ is $1$ if the $i$th nucleotide of $A$ is unpaired, or $0$ if it is paired.  The auxiliary data pseudoenergy term is motivated by the Zarringhalam et al. model from \cite{Zarringhalam:2012}.  Here, the auxiliary data is rescaled to the range $[0, 1]$, so that $|x_i - M_i|$ is an agreement score between the structure and the data. Details on our data model are below.

Although the Nussinov-Jacobson model is much simpler than the models used to predict secondary structures, it will still illustrate the problems inherent to other pseudoenergy models.  In particular, the parameter $C$ determines how much the auxiliary data should be weighted relative to the underlying energies assigned to each structure.  The Deigan et al. model, Zarringhalam et al. model, and Rsample have similar parameters.  We will see that when $C$ is small, one structure may dominate over the other structure in a bimodal mixture because the pseudoenergy is not strong enough to overcome differences in thermodynamic energy.  On the other hand, if $C$ is large, the pseudoenergy is large enough that the Boltzmann distribution is collapsed towards a single structure.

In order to analyze the Nussinov-Jacobson model, we make two simplifications.  First, we remove all noise from the auxiliary data.  That is, the auxiliary data sequence corresponding to a structure $A$ will be a sequence of zeros and ones, where a zero represents that the nucleotide is paired, and a one represents that the nucleotide is unpaired.  Then, we will consider selecting approximate structures at random as follows: for each nucleotide in a sequence, we will choose whether it is paired or unpaired independently of all other nucleotides with probability $q \in (0, 1)$.  Choosing $q \approx 0.6$ corresponds to the approximate percentage of paired nucleotides in \textsc{rna} sequences, where choosing $q = 0.5$ would give all structures equal probability.  We will show the following:
%(Note that the exact value of $p$ does not make a substantial difference in the results below.)

%\begin{thm} \label{Probabilistic} Consider pairs of RNA structures where each nucleotide is paired or unpaired independently with probability $0.6$.  Choose any $n$ between 20 and 300.  Then, for any choice of parameter $C$ in the Nussinov-Jacobson model, for over $50\%$ of pairs of structures of length $n$, there exists a conformational ratio where the total variation distance between the correct and predicted distributions is at least $0.25$.
%\end{thm}

\begin{thm} \label{ProbabilisticLimit} Fix $q \in (0, 1)$ and fix $C > 0$ in the Nussinov-Jacobson model.  Consider \textsc{rna} structures where each nucleotide is paired independently with probability $q$.  Let $P_n$ be the probability that two randomly chosen structures of length $n$ have a conformational ratio that the Nussinov-Jacobson model cannot reconstruct with total variation distance less than $0.25$.  Then, $\lim_{n \to \infty} P_n = 1$.
\end{thm}

This indicates that regardless of the parameterization of the Nussinov-Jacobson model, almost all structures of long enough length have ratios that cannot be reconstructed even when the auxiliary data is noiseless.  In Figure \ref{fig:Nussinov} below, we will show that even for small values of $n$, $P_n$ is large.  Although a similar analysis becomes much more complicated for other pseudoenergy models, there are examples of structures for the Deigan et al. and Zarringhalam et al. models where conformational ratios cannot be recovered regardless of the parameters in the models.

To quantify the problems in the Nussinov-Jacobson model when $C$ is large or small, we turn back to the simulated experiments from Section \ref{Sec:Examples}.  We saw in Figure \ref{fig:Bistable} that for Rsample, as $p$ increases, the predicted conformational ratio in the sample increases nearly linearly, which is the target behavior we would hope to see for all multimodal distributions.  However, in the \textsc{hiv-1 rra} samples in Figure \ref{fig:HIV}, the red structure consistently dominates over the blue structure for the majority of $p$ values.  We aim to characterize this problem mathematically.

Consider a general \textsc{rna} sequence with two stable structures $A$ and $B$.  In this section, we model a noiseless data sequence $\{S_i\}$ by defining $S_i = 0$ if the $i$th nucleotide in $A$ is paired, and $1$ if it is unpaired.  We similarly define a binary auxiliary data sequence $\{T_i\}$ for $B$.  Then, as before, we let $M(p) = \{p S_i + (1 - p) T_i\}$ be an auxiliary data sequence corresponding to a mixture of ratio $p$ of structure $A$ and ratio $(1 - p)$ of structure $B$.  We define $\bp(A)$ to be the number of nucleotides paired in $A$ as above, and also $\bp(A - B)$ to be the number of nucleotides paired in $A$ that are not paired in $B$.  Finally, we let $\bp(A \Delta B) = \bp(A - B) + \bp(B - A)$.

We define the \emph{crossover point} for structures $A$ and $B$ to be the value of $p$ where two structures appear with equal probability in the predicted Boltzmann distribution.  In an accurate prediction, the crossover point should be near $p = 0.5$, but for Rsample in the \textsc{hiv-1 rra} example, it is above $p = 0.7$.  This means that there is a weighting of the \textsc{hiv-1 rra} structures where the total variation distance between the correct and predicted distributions is at least $0.2$.

Next, we see that the range of values where both structures appear in Figure \ref{fig:HIV} is small: for Rsample, one of the two structures was barely present in all proportions $p \leq 0.3$.  Thus, we define the \emph{crossover window} for structures $A$ and $B$ to be the range of $p$ values for which both structures appear with probability at least $0.1$.  We would expect the crossover window to be of width approximately $0.8$ (corresponding to the range $0.1 \leq p \leq 0.9$), but it is less than $0.5$ in this example.  Below, we find that a short crossover window leads to high total variation distances.  One advantage of using crossover windows to bound total variation distances is that the ratio of probabilities of structures $A$ and $B$ is independent of the partition function in the Boltzmann distribution.

We will see that when $C$ is relatively small, this can lead to uncentered crossover points, because the data pseudoenergies are not strong enough to overcome the energy differences between competing structures.  On the other hand, when $C$ is large, a small change in the data can lead to large changes in energy, which will result in short crossover windows.  Now, we find two simple results about crossover points and crossover windows under the Nussinov-Jacobson model:

\begin{lem} \label{NJCrossoverPoint}
Consider an \textsc{rna} sequence with two structures $A$ and $B$ that have corresponding binary auxiliary data sequences $\{S_i\}$ and $\{T_i\}$.  In the Nussinov-Jacobson \textsc{shape}-directed energy model, the crossover point for structures $A$ and $B$ is given by
\[
p^* = \frac{1}{2} + \frac{1}{2C} - \frac{\bp(A - B)}{C\bp(A \Delta B)}.
\]
%In particular, for all $C \leq 5$ and all structures $A$ and $B$ where $\bp(A - B) = 0$, we have $p^* \geq 0.6$.
\end{lem}

\begin{proof}
%Here, we must analyze the penalties of the form $C|x_i - M_i|$ in Equation \eqref{Nussinov}. 
 Each nucleotide counted in $\bp(A - B)$ contributes $C(1 - p)$ to the energy of $A$ because the nucleotide is paired in $A$ (so $x_i = 0$), and the auxiliary data value at this nucleotide is $(1 - p)$, because $B$ has an unpaired nucleotide here.  Similarly, each nucleotide counted in $\bp(B - A)$ also contributes $C(1 - p)$ because $x_i = 1$ and the auxiliary data value is $p$.  Thus,
\[
\ENJ(A | M(p)) = -\bp(A) + \bp(A - B)C(1 - p) + \bp(B - A)C(1 - p).
\]
A similar equation holds for $B$, and the crossover point $p^*$ satisfies the equation where the energies for $A$ and $B$ are equal.  This gives:
\[
p^* = \frac{1}{2} + \frac{\bp(B) - \bp(A)}{2C\big[\bp(A - B) + \bp(B - A)\big]}.
\]
The set-theoretic relation $\bp(A) + \bp(B - A) - \bp(A - B) = \bp(B)$ can be used to simplify the equation for $p^*$ to the one given in the lemma.
%Finally, when $\bp(A - B) = 0$, we have the simplification $p^* = \frac{1}{2} + \frac{1}{2C}$, which is at least $0.6$ when $C \leq 5$. %\qed
\end{proof}

\begin{lem} \label{NJCrossoverWindow}
Consider an \textsc{rna} sequence with two structures $A$ and $B$ that have corresponding binary auxiliary data sequences $\{S_i\}$ and $\{T_i\}$.  In the Nussinov-Jacobson \textsc{shape}-directed energy model, the crossover window for structures $A$ and $B$ has length at most 
\[
\frac{RT \ln 9}{C\bp(A \Delta B)}.
\]
\end{lem}

\begin{proof}
In order to have at least 10\% of each structure appearing in the \textsc{shape}-directed distribution, a necessary (but insufficient) pair of conditions is:
\begin{align} 
\mathbb{P}(A | M(p) ) \geq 0.1 \geq \frac{1}{9} \mathbb{P}(B | M(p)), \nonumber \\
\mathbb{P}(B | M(p)) \geq \frac{1}{9} \mathbb{P}(A | M(p)). \label{ProbabilityCondition}
\end{align}
The advantage of these conditions is that by including probabilities on both sides of the inequality, the partition function $Z$ can be cancelled from the inequality entirely.  Inserting the expressions for $\ENJ(A | M(p))$ and $\ENJ(B | M(p))$ from the proof of Theorem \ref{NJCrossoverPoint} into Equation \eqref{ProbabilityCondition} and taking natural logarithms yields the equivalent set of inequalities,
\begin{equation*} \label{ProbIneq}
\bp(B) - \bp(A) - RT\ln 9 \leq C(2p - 1) \bp(A \Delta B) \leq \bp(B) - \bp(A) + RT\ln 9.
\end{equation*}
In order to find an upper bound for the length of the crossover window, we must find how slowly the expression $C(2p - 1) \bp(A \Delta B)$ can increase from the lower limit to the upper limit, as a function of $p$.  Recognizing that the difference between the lower and upper limits is $2RT \ln 9$, and letting $\ell$ be the length of the crossover window, this is equivalent to:
\[
2C\bp(A \Delta B) \ell \leq 2RT\ln 9,
\]
which can be rearranged to finish the proof. %\qed
\end{proof}

With these facts about the crossover point and crossover window, we are ready to show that the Nussinov-Jacobson model fails to reconstruct conformational ratios for most randomly chosen pairs of structures.

\begin{proof}[Proof of Theorem \ref{ProbabilisticLimit}]
We will break into cases, depending on whether the parameter $C \leq C^*_n$ or $C > C^*_n$ for a cutoff value $C^*_n$ that depends on $n$, described further later.  If $C$ is large, then small changes in experimental data lead to big changes in predicted conformational ratios, causing the crossover window to be narrow.  On the other hand, if $C$ is small, the data is not weighted heavily, causing uncentered crossover points.  As we will see, both of these conditions will lead to large maximum total variation distances.

Let $C_q$ be any positive constant less than $2q(1-q)$.  Then, based on an approximation of the binomial distribution with a normal distribution (used below), we choose the cutoff value $C_n^* = \frac{RT \ln 9}{0.3C_q n}$.

Fix any small $\epsilon > 0$.  We will show that for $n$ sufficiently large, the probability of two randomly chosen structures having a conformational ratio where the total variation distance is greater than $0.25$ is at least $1 - \epsilon$.\\
\hrule\ \\
Case 1: $C \leq C_n^*$.  In this case, we will argue that the crossover point is uncentered with high probability, which will lead to a large total variation distance with high probability.  In particular, if the crossover point $p^* \geq 0.75$ for structures $A$ and $B$, then the total variation distance must be at least $0.25$ at the weighting $p = p^*$.  By Lemma \ref{NJCrossoverPoint}, $p^* \geq 0.75$ is equivalent to:
\[
\bp(A - B) \leq \left(\frac{1}{2C} - \frac{1}{4} \right) C \bp(A \Delta B).
\]
Then, let $i$ represent the number of nucleotides where the two structures are both paired or both unpaired, so that $\bp(A \Delta B) = n - i$.  If we condition on knowing $i$ in advance, then we can see that $p^* \geq 0.75$ is equivalent to:
\[
\bp(A - B) \leq \left(\frac{1}{2} - \frac{C}{4}\right) (n - i).
\]
Note that for all $0 \leq i \leq n$, the right hand side is nonpositive and nonincreasing in $C$.  Thus, this inequality is most restrictive on $\bp(A - B)$ when $C$ is the largest possible value, which we have assumed is when $C = C_n^*$.   Therefore,
\[
\mathbb{P}(p^* \geq 0.75) \geq \sum_{i = 0}^{n} \mathbb{P}(\bp(A \Delta B) = n - i) \cdot \mathbb{P}\left(\bp(A - B) \leq \left( \frac{1}{2} - \frac{C_n^*}{4}\right)(n - i)\right).
\]
If we let $j = \bp(A - B)$, we can use our model where a nucleotide is paired or unpaired independently with probability $q$ to rewrite the sum:
\begin{equation} \label{eq:CPProbBound}
\mathbb{P}(p^* \geq 0.75) = \sum_{i = 0}^{n} \binom{n}{i} (q^2 + (1 - q)^2)^i (q(1 - q))^{n - i} \sum_{j = 0}^{\left\lfloor \left(\frac{1}{2} - \frac{C_n^*}{4}\right)(n - i)\right\rfloor}  \binom{n - i}{j}.
\end{equation}
We will show that the inner sum is at least $2^{n - i - 1} \left(1 - \frac{\epsilon}{2}\right)$ for all $i \leq n$, for $n$ sufficiently large.  Applying this inequality will allow us to approximate the remaining sum using a normal distribution.  In order to show this, we will interpret the inner sum as a scaled cumulative distribution function for a binomial distribution.  Let $k = n - i$.  We have:
\begin{equation} \label{FloorEquation}
\sum_{j = 0}^{\left\lfloor \left(\frac{1}{2} - \frac{C_n^*}{4}\right)(n - i)\right\rfloor}  \binom{n - i}{j} = 2^{k} \sum_{j = 0}^{\left\lfloor \left(\frac{1}{2} - \frac{C_n^*}{4}\right)k\right\rfloor}  \binom{k}{j} \left(\frac{1}{2}\right)^j \left(\frac{1}{2}\right)^{k - j}
\end{equation}
Now, let $\{X_m\}$ be a sequence of independent and identically distributed Bernoulli random variables with parameter $1/2$.  Then,
\begin{align*}
\sum_{j = 0}^{\left\lfloor \left(\frac{1}{2} - \frac{C_n^*}{4}\right)k\right\rfloor}  \binom{k}{j} \left(\frac{1}{2}\right)^j \left(\frac{1}{2}\right)^{k - j} &= \mathbb{P}\left(X_1 + \ldots + X_k \leq \left(\frac{1}{2} - \frac{C_n^*}{4}\right)k \right)\\
&= \mathbb{P} \left( \frac{X_1 + \cdots + X_k - k/2}{\sqrt{k}/2} \leq \frac{-C_n^*\sqrt{k}}{2}\right)
\end{align*}
Here, the Central Limit Theorem implies that this renormalized sum of Bernoulli random variables approaches a standard normal distribution (where we used that the mean and standard deviation of each $X_m$ are both $1/2$).  Also, because the Bernoulli distribution has moments of all orders, we can apply the Berry-Esseen theorem to get a bound on how quickly this probability approaches the approximation given by the normal distribution (see, for example, Theorem 3.4.17 in \cite{Durrett:2019}).  Let $B$ be the constant from the Berry-Esseen theorem that depends only on the parameter of the Bernoulli random variables, $1/2$, and let $\Phi$ be the cumulative distribution function for a normal random variable with mean $0$ and standard deviation $1$.  Then, we have:
\[
\mathbb{P} \left( \frac{X_1 + \cdots + X_k - k/2}{\sqrt{k}/2} \leq  \frac{-C_n^*\sqrt{k}}{2} \right) \geq \Phi\left(  \frac{-C_n^*\sqrt{k}}{2} \right) - \frac{B}{\sqrt{k}}
\]
Since $k  = n - i \leq n$, and since $C_n^* = \frac{RT \ln 9}{0.3C_q n} = O\left(n^{-1}\right)$ as $n$ approaches infinity, we see that the input to $\Phi$ approaches $0$ as $n$ grows large.  Additionally, if we assume $k \geq C_\epsilon$ for some large constant $C_\epsilon$ that depends on $\epsilon$ but is independent of $n$, then we can assume:
\[
\Phi\left(  \frac{-C_n^*\sqrt{k}}{2} \right) - \frac{B}{\sqrt{k}} \geq \Phi(0) \left(1 - \frac{\epsilon}{2}\right) = \frac{1}{2}\left(1 - \frac{\epsilon}{2}\right)
\]
Tracing back to Equation \eqref{FloorEquation}, we have shown for $n$ sufficiently large and for $k > C_\epsilon$,
\begin{equation}\label{eq:ApproxBinom}
\sum_{j = 0}^{\left\lfloor \left(\frac{1}{2} - \frac{C_n^*}{4}\right)k\right\rfloor}  \binom{k}{j} \geq 2^{k - 1} \left(1 - \frac{\epsilon}{2}\right).
\end{equation}
Recall that $k = n - i$, so that $k > C_\epsilon$ is equivalent to $i < n - C_\epsilon$.  Thus, substituting Equation \eqref{eq:ApproxBinom} into Equation \eqref{eq:CPProbBound} when $i < n - C_\epsilon$, and dropping the terms where $i \geq n - C_\epsilon$, gives us:
\begin{align*}
\mathbb{P}(p^* \geq 0.75) &\geq \sum_{i = 0}^{\lfloor n - C_\epsilon - 1\rfloor} \binom{n}{i} (q^2 + (1 - q)^2)^i (q(1 - q))^{n - i} 2^{n - i - 1}\left(1 - \frac{\epsilon}{2}\right)\\
&= \frac{1}{2}\left(1 - \frac{\epsilon}{2}\right) \sum_{i = 0}^{\lfloor n - C_\epsilon - 1\rfloor} \binom{n}{i} (q^2 + (1 - q)^2)^i (2q(1 - q))^{n - i}\\
&= \frac{1}{2}\left(1 - \frac{\epsilon}{2}\right) \left(1 - \sum_{i = \lfloor n - C_\epsilon\rfloor}^{n} \binom{n}{i} (q^2 + (1 - q)^2)^i (2q(1 - q))^{n - i}\right),
\end{align*}
where the last line is true because the summation again represents a binomial distribution, this time with parameter $q^2 + (1 - q)^2$.  Because $C_\epsilon$ is independent of $n$, the last sum approaches $0$ as $n$ becomes large.  For this reason, for $n$ sufficiently large,
\[
\mathbb{P}(p^* \geq 0.75) \geq \frac{1}{2} \left(1 - \epsilon\right).
\]
Finally, doubling through symmetry, we see that if $C \leq C_n^*$, then for all $n$ sufficiently large,
\[
\mathbb{P}(p^* \not \in (0.25, 0.75)) \geq 1 - \epsilon,
\]
proving in this case, $P_n \geq 1 - \epsilon$.\\
\hrule\ \\
Case 2: Now, assume that $C \geq C_n^*$.  Our goal will be to show that the crossover window for any two random structures has width at most $0.3$ with probability at least $1 - \epsilon$.  Note that when the crossover window has width at most $0.3$, then this forces the maximum total variation distance to be at least $0.25$:  in the best-case scenario, the crossover window would be centered around $p = 0.5$, in which case the total variation distance at both $p = 0.35$ and $p = 0.65$ is at least $0.25$.

Now, let $\ell$ be the width of the crossover window for randomly selected structures $A$ and $B$.  Then, we know from Lemma \ref{NJCrossoverWindow} that
\[
\ell \leq \frac{RT \ln 9}{C \bp(A \Delta B)},
\]
and so we will set
\[
\frac{RT \ln 9}{C \bp(A \Delta B)} \leq 0.3.
\]
This is equivalent to
\[
\bp(A \Delta B) \geq \frac{RT \ln 9}{C \cdot 0.3}.
\]
Using that $C \geq C_n^* = \frac{RT \ln 9}{0.3C_q n}$, we have:
\begin{align*}
\mathbb{P}(\ell \leq 0.3) &\geq \mathbb{P}\left(\bp(A \Delta B) \geq \frac{RT \ln 9}{C \cdot 0.3}\right)\\
&\geq \mathbb{P}\left(\bp(A \Delta B) \geq C_q n\right)\\
&= 1 - \mathbb{P}\left(\bp(A \Delta B) < C_q n\right)\\
&= 1 - \sum_{i = 0}^{\left\lfloor C_q n\right\rfloor} \binom{n}{i} \left(2q(1 - q)\right)^i \left(q^2 + (1 - q)^2\right)^{n - i}\\
\end{align*}
Now, we repeat the procedure from Case 1 to analyze the remaining sum: let $\{Y_m\}$ be a sequence of independent Bernoulli random variables with parameter $2q(1-q)$, and standard deviation $\sigma_q$.  Then,
\begin{multline*}
\sum_{i = 0}^{\left\lfloor C_q n \right\rfloor} \binom{n}{i} \left(2q(1 - q)\right)^i \left(q^2 + (1 - q)^2\right)^{n - i} = \mathbb{P}\left(Y_1 + \cdots + Y_n \leq  C_q n \right)\\
= \mathbb{P} \left(\frac{Y_1 + \cdots + Y_n - n\big(2q(1-q)\big)}{\sigma_q \sqrt{n}} \leq \frac{C_q n - n\big(2q(1-q)\big)}{\sigma_q \sqrt{n}}\right)\\
\leq \Phi \left(\frac{ \sqrt{n} \big(C_q - 2q(1-q)\big)}{\sigma_q}\right) + \frac{B_q}{\sqrt{n}},
\end{multline*}
where the last line is again due to the Berry-Esseen theorem, and the constant $B_q$ is a potentially larger constant than $B$ in Case 1 that depends only on the parameter of the Bernoulli random variables, $\{Y_m\}$.  Here, by our choice of $C_q < 2q(1-q)$, we see that the input to $\Phi$ approaches negative infinity as $n$ approaches infinity.  Additionally, the second term is arbitrarily small as $n$ increases.  Thus, we can conclude that for $n$ sufficiently large,
\[
\sum_{i = 0}^{\left\lfloor C_q n\right\rfloor} \binom{n}{i} \left(2q(1 - q)\right)^i \left(q^2 + (1 - q)^2\right)^{n - i} < \epsilon,
\]
from which we can conclude that for all $C \geq C_n^*$, for $n$ sufficiently large,
\[
\mathbb{P}(\ell \leq 0.3) \geq 1 - \epsilon.
\]
Again, this implies that when $C \geq C_n^*$, $P_n \geq 1 - \epsilon$.  Thus, between Case 1 and Case 2, we found sufficient conditions for small and large $C$ where the total variation distance was above 0.25 with high probability, which completes the proof.

\end{proof}

Theorem \ref{ProbabilisticLimit} tells us that the Nussinov model is unsuccessful at recreating conformational ratios for long \textsc{rna} sequences.  We can also use the same proof structure to find lower bounds on $P_n$ for small $n$.   The challenging part of such an analysis is picking the optimal cutoff $C_n^*$ between the two cases.  As the cutoff increases, the probability in Case 1 increases while the probability in Case 2 decreases.  However, the probabilities in Case 2 are a step function of the cutoff, indicating that there are only finitely many cutoff values to consider for each value of $n$ to optimize the bound.  Figure \ref{fig:Nussinov} shows a lower bound for $P_n$ for $n$ from $1$ to $300$ when $q = 0.6$ (corresponding to approximately the percentage of pairings in \textsc{rna} structures).  By length $20$, over half of the pairs of structures have a conformational ratio that cannot be reconstructed accurately.

\begin{figure}
\begin{center}
\includegraphics[width=0.7\textwidth]{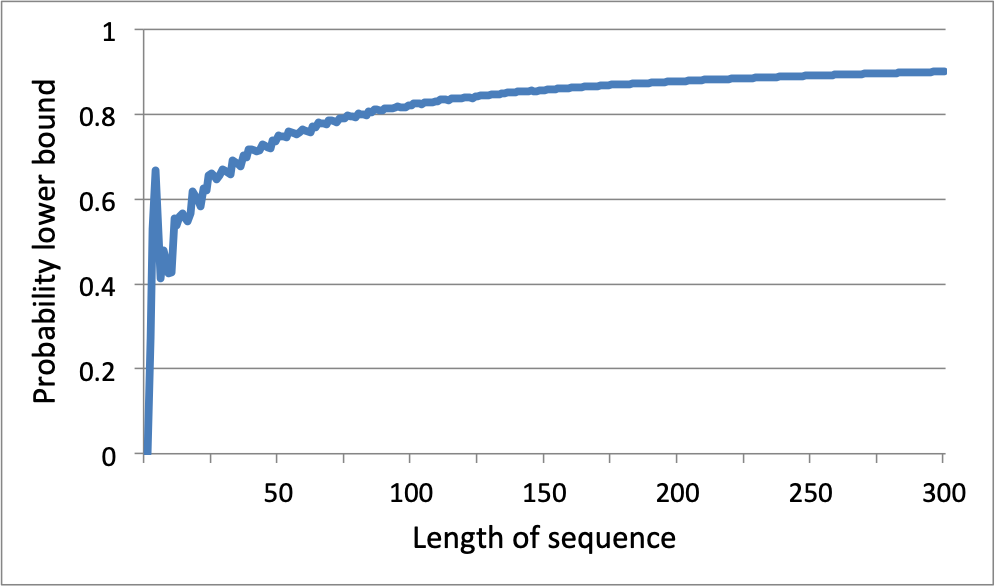}
\end{center}
\caption{The plot above gives a lower bound for the probability that a randomly selected pair of structures of length $n$ has a conformational ratio that cannot accurately be reconstructed by the Nussinov-Jacobson model.}
\label{fig:Nussinov}
\hrule 
\end{figure}

\section{Information content in auxiliary data} \label{Sec:Information}

We have found evidence that pseudoenergy-based prediction models often cannot reconstruct the weightings in a bimodal distribution of structures, even if the data is noiseless.  But, how much weighting information is contained in the data itself?  In this section, we will quantify this by using Fisher information.

In examining the information in the data, we separate the prediction of structures from the prediction of weightings.  This is motivated by the ``sample and select'' algorithm from \cite{Quarrier:2010}, where a sample of putative structures is generated from the undirected Boltzmann distribution, and then the structure which agrees most closely with the auxiliary data is selected.  For multimodal distributions, this method would need to be refined so that a set of putative structural modes is selected, and then the auxiliary data is used to predict the conformational ratios between these structures.

Here, we assume that the structures in a bimodal distribution are already known in advance.  This represents the best-case scenario, where the only unknown in a distribution of structures is the conformational ratio.  While this may sound like a strong assumption, signals from the component structures of a distribution often appear in the undirected Boltzmann distribution.  In particular, RNAprofiling, \cite{Rogers:2014}, revealed signals from both structures for the Bistable and \textsc{hiv rra} sequences.  Additionally, as we have seen in Section \ref{Sec:Examples} above, Rsample, \cite{Spasic:2018}, is often able to identify the correct structures within a distribution, regardless of whether the conformational ratio is correct.

Thus, given a sequence of auxiliary data corresponding to a mixture of two known structures $A$ and $B$, we aim to reliably predict the conformational ratio $p$ of structure $A$ in the distribution. We will show first through example, and then rigorously via Fisher information, that the ability to recover $p$ depends on the number of locations where structures $A$ and $B$ differ.  To the best of our knowledge, this is the first analysis of how the ability to predict $p$ depends on the structures in a distribution.

%In order to evaluate the information in the data, we separate the prediction of structures from the prediction of weightings.  Here, we assume that the structures in a bimodal distribution are already known in advance.  While this may sound like a strong assumption, Rsample, \cite{Spasic:2018}, is often able to identify the correct structures within a distribution, regardless of whether the weighting is correct.  Additionally, the ``sample and select'' algorithm from \cite{Quarrier:2010} predicts structures by using a pool of potential structures from an undirected Boltzmann distribution before incorporating \textsc{shape} data. Starting with two known structures and  \textsc{shape} data representing a mixture in an unknown ratio, we check whether we can recover the ratio.

%\hrule\ \\
\subsection{Example} Consider the \textsc{hiv-1 rra} sequence from Section \ref{Sec:Examples}.  We use the same simulated sequences $\{T_i\}$ and $\{S_i\}$ for each of the two known \textsc{hiv-1 rra} structures.  Then, as in Section \ref{Sec:Examples}, we mix these sequences together with various proportions, $p$, to get a new \textsc{shape} sequence $M(p) =  \{pS_i + (1 - p) T_i\}$.

Now, with only the numerical values of the sequence $M(p)$ and the two \textsc{hiv-1 rra} structures, we check whether we can recover the value, $p$.  We use only the values from $M(p)$ corresponding to where structures $A$ and $B$ differ because the \textsc{shape} values come from different distributions at these points, and are expected to carry the most information about $p$.  Each of these values is drawn from a convolution of two different known distributions, with an unknown weighting.  A common way to predict the weighting is to use a maximum likelihood estimator, which finds the value $\hat p$ which maximizes the probability of producing the observed auxiliary data sequence.  In summary, we simulated sequences $\{T_i\}$ and $\{S_i\}$, mixed them in different ratios $p$, and used the maximum likelihood estimator $\hat p$ to estimate $p$ from $M(p)$.   The results are shown in Figure \ref{fig:MLE}.

Here, the estimator successfully recovered the $p$ values with a maximum error of $0.11$, which is an improvement over the pseudoenergy models in Figure \ref{fig:HIV}.  However, if we repeat the same procedure on the 16S r\textsc{rna} Four way Junction from the test set in \cite{Spasic:2018}, then the maximum error is $0.24$.   This can be explained by enumerating the structural differences between the modes: the two \textsc{hiv-1 rra} structures differ in 31 positions, while the two 16S r\textsc{rna} structures differ in only 16 positions.
\begin{figure}
\begin{center}
\includegraphics[width=0.9\textwidth]{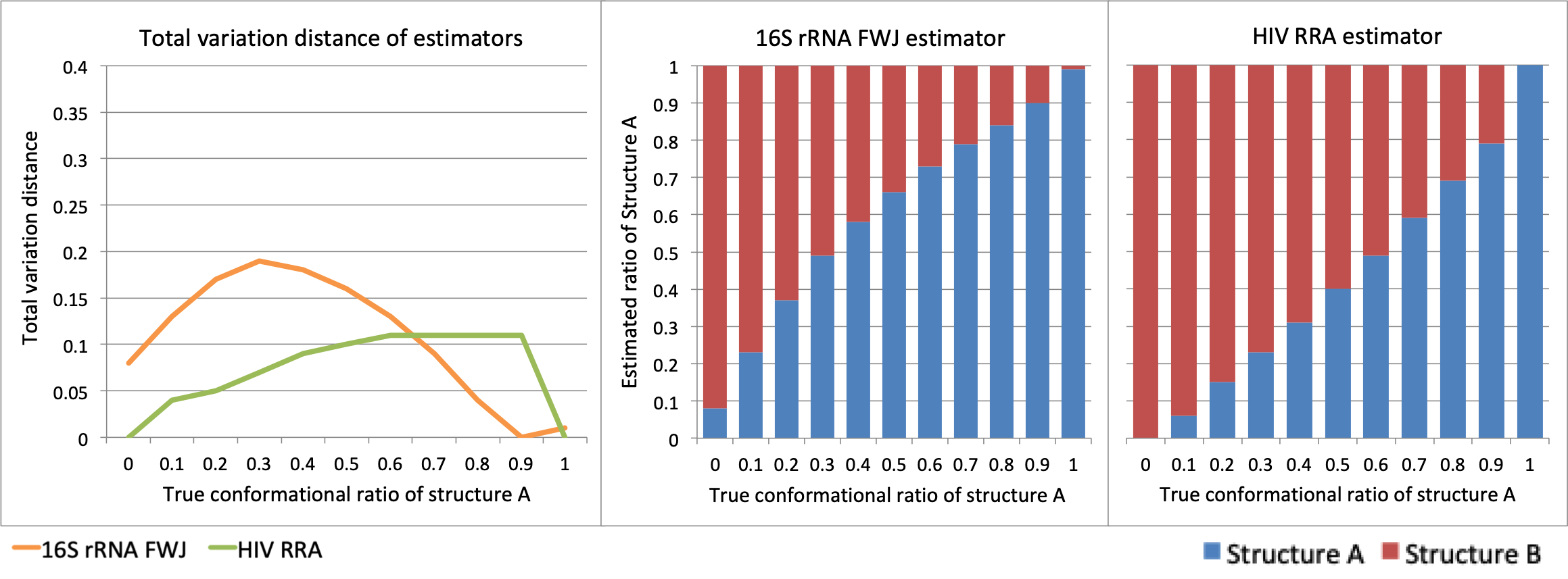}
\end{center}
\caption{The maximum error for the 16S r\textsc{rna} Four way Junction is higher than the maximum error for the \textsc{hiv rra} sequence.  This is because there are fewer structural differences between the 16S structures than the \textsc{hiv-1 rra} structures.}
\label{fig:MLE}
\hrule
\end{figure}

\subsection{Fisher information}  
To assess how accurate an estimator $\hat p$ can recover $p$, we study the variance of $\hat p$.  The more the structures in a distribution differ, the more information the auxiliary data is expected to carry about $p$, which should lead to a lower variance of $\hat p$.  
%This will lower the variance of $\hat p$, leading to more confident predictions of $p$.  We focus on positions where structures $A$ and $B$ are not both paired or both unpaired because at these points, the \textsc{shape} data comes from a mixture of two distinct distributions.  As a result, we expect the most information about $p$ to come from these positions.  
This means that the variance of $\hat p$ is not directly a function of the length of the \textsc{rna} sequences, but is instead a function of the number of locations where the structures differ.

To quantify the relationship between the variance of an estimator and the number of structural differences between modes, we use Fisher information.  The Fisher information of a random variable, defined below, describes how much information the random variable contains about a parameter.  By using Fisher information and the model of \textsc{shape} data from \cite{Sukosd:2013}, we will be able to conclude statements like the following:
\begin{lem}\label{FI}
Consider a sequence of $n$ simulated \textsc{shape} values derived from an equal mixture of two structures, where at each position, one structure is center-paired and the other structure is unpaired.  Let $\hat p$ be any unbiased estimator of the true conformational ratio, $p = 0.5$.  Then, $\Var(\hat p) \geq \frac{1}{2.15n}$.
\end{lem}

To prove Lemma \ref{FI}, consider a single \text{shape} value $X = pX_1 + (1 - p)X_2$ corresponding to a nucleotide that is unpaired with \textsc{shape} value $X_1$ in proportion $p$ of structures, and center-paired with \textsc{shape} value $X_2$ in proportion $1 - p$ of structures.  We choose to compare these distributions because they are the most different, and represent the scenario where $p$ should be recovered most easily.  We describe how to modify the approach to include edge-paired nucleotides below.  

The Fisher information is a measure of how much information a single random observation of $X$ gives about a fixed parameter $p$ on average.  Some values of $X$ give less information about $p$ than others.  For example, if $X$ is a middling value, it is unclear whether $X$ is the average of a paired and unpaired signal, or whether $X$ is a lower-than-usual unpaired signal.

Formally, let $g(x, p)$ be the probability density function of the random variable $X$ with parameter $p$.  For a fixed value of $x$ but varied values of $p$, if $g(x, p)$ has a sharp maximum, then a maximum likelihood estimator should be able to recover $p$ more easily.  This is quantified by the Fisher information:
\begin{equation} \label{eq:FisherInf}
I(p) := \mathbb{E} \left[ \left. \left(\frac{\partial}{\partial p} \log g(X, p)\right)^2 \right| p \right] = \int_{-\infty}^\infty \left(\frac{\partial}{\partial p} \log g(x, p)\right)^2 g(x, p) dx.
\end{equation}
The Cram\'{e}r-Rao bound illustrates the utility of Fisher information: for any unbiased estimator $\hat p$ of a parameter $p$,
\[
\Var \hat p \geq \frac{1}{I(p)}.
\]
Additionally, when collecting independent samples, the Fisher information is additive.  Therefore, if we have a sample of $n$ \textsc{shape} values corresponding to nucleotides that are unpaired in proportion $p$ of structures, and center-paired in proportion $1 - p$ of structures, then for any unbiased estimator $\hat p$ of $p$,
\[
\Var \hat p \geq \frac{1}{nI(p)}.
\]
Lemma \ref{FI} follows from calculating the Fisher information for $p = 0.5$.  Typically, the maximum likelihood estimator for a parameter is not unbiased, but under mild regularity conditions, it asymptotically approaches an unbiased estimator with minimal variance, given by $\frac{1}{nI(p)}$ as $n$ grows large (see, for example, Theorem 9.18 in \cite{Wasserman:2004}).  

The Fisher information could also be used to describe the variance of an estimator when the true value is not $p = 1/2$, by keeping track of more details about the structural differences in the modes.  If there are $k$ positions where structure $A$ is unpaired and structure $B$ is paired, and $\ell$ positions where where structure $A$ is center-paired and structure $B$ is unpaired, and if structure $A$ has conformational ratio $p$, then
\[
\Var \hat p \geq \frac{1}{kI(p) + \ell I(1 - p)}.
\]
This could be extended to include edge-paired nucleotides by considering more terms.

Now, we compute the Fisher information.  Since $g(x, p)$ is the density function for a weighted sum, it is the weighted convolution of the paired and unpaired \textsc{shape} distributions:
\[
g(x, p) = \int_0^{y + 0.025(1 - p)} \frac{1}{p(1 - p)} f_{\mbox{\scriptsize un}}\left(\frac{y}{p}\right) f_{\mbox{\scriptsize pair}}\left(\frac{x - y}{1 - p}\right) dy.
\]
Here, $f_{\mbox{\scriptsize un}}$ and $f_{\mbox{\scriptsize pair}}$ represent the unpaired and center-paired \textsc{shape} probability distribution functions from \cite{Sukosd:2013}, respectively.  $f_{\mbox{\scriptsize pair}}(x)$ is a generalized extreme value distribution with parameters $\xi = 0.762581, \sigma = 0.0492536,$ and $\mu = 0.0395857$.  $f_{\mbox{\scriptsize un}}(x)$ is an exponential distribution with parameter $\lambda = 1.46797$.  The domain of integration is the support of the integrand.

To evaluate the Fisher information, we note:
\[
\frac{\partial}{\partial p} \ln g(x, p) = \frac{1}{g(x, p)} \cdot \frac{\partial}{\partial p} g(x, p).
\]
Because  $f_{\mbox{\scriptsize un}}$ and $f_{\mbox{\scriptsize pair}}$ are continuous and have continuous derivatives over their support, and because the bounds of the convolution $g(x, p)$ are continuous with respect to $p$, we can use the Liebniz rule to evaluate the derivative of $g(x, p)$:
\begin{multline*}
\frac{\partial}{\partial p} g(x, p) = \frac{d}{dp} \left(\frac{1}{p(1 - p)} \right) \int_0^{y + 0.025(1 - p)} f_{\mbox{\scriptsize un}}\left(\frac{y}{p}\right) f_{\mbox{\scriptsize pair}}\left(\frac{x - y}{1 - p}\right) dy \\
+ \frac{1}{p(1 - p)}  \int_0^{y + 0.025(1 - p)}  \frac{\partial}{\partial p} \left( f_{\mbox{\scriptsize un}}\left(\frac{y}{p}\right) f_{\mbox{\scriptsize pair}}\left(\frac{x - y}{1 - p}\right)\right) dy
\end{multline*}

In its general form, the Liebniz rule also involves terms where the integrand is evaluated at the bounds of the integral, but in this case, these terms are zero.  The second term above can be split into two integrals after applying the product rule, allowing us to write $\frac{\partial}{\partial p} g(x, p)$ as the sum of three integrals.  Plugging this expression into \eqref{eq:FisherInf} and using an iterated numerical integration in Matlab, we evaluated the Fisher information, and the results are shown in Figure \ref{fig:FI}.
\begin{figure}
\begin{center}
\includegraphics[width=0.7\textwidth]{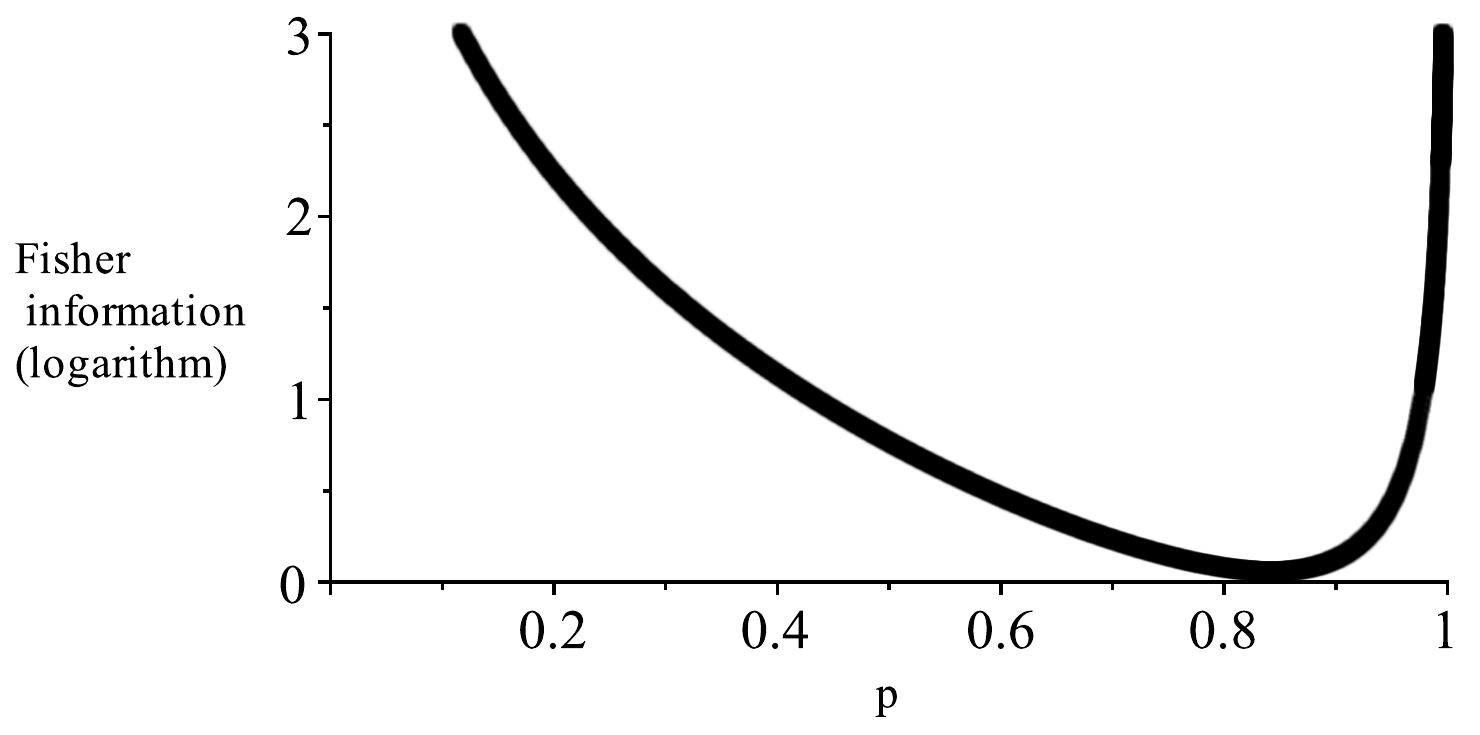}
\end{center}
\caption{For each value of $p$, the Fisher information $I(p)$ was calculated by using numerical integration in Matlab.  The vertical axis is on a logarithmic scale.}
\label{fig:FI}
\hrule
\end{figure}

$I(0.5) = 2.15$, yielding the bound in Lemma \ref{FI}.  This can be leveraged to determine how many structural differences two structural modes need to have in order to accurately recover $p$ values.  If we would like the standard deviation of $\hat p$ to be less than $0.1$, for example, the number of differences $n \geq 47$.  In other words, it is only possible to recover $p$ with this level of certainty if our two structures have at least $47$ positions where they differ.

\section{Conclusion}

Recent advances to pseudoenergy models such as Rsample, \cite{Spasic:2018}, are able to identify multiple structural modes within a distribution.  However, it is still challenging to reliably reconstruct the conformational ratios of the modes.  One issue is that long \textsc{rna} sequences tend to have very different structures with similar energies.  When this happens, a short crossover window leads pseudoenergy models to inaccurate predictions.  An alternative approach is to separate predictions of structural modes from their conformational ratios.  When the structural modes are identified in advance, more accurate conformational ratios can be found by using estimators from statistics.  

The variance of the estimator depends on how many differences there are between the structural modes.  In practice, when predicting an unknown conformational weighting between structural modes, the more the modes differ, the higher the confidence in the prediction.  Fisher information can be used to give an estimated confidence level when using a maximum likelihood estimator to recover the conformational ratio.

%In line with the Medloop example above, we see that when $p$ is near $0.5$, the average amount of information gained by a single \textsc{shape} value is much less than the information when $p$ is near $0$ or $1$.  In fact, Fisher information is additive over independent draws of random variables, and the variance of the maximum likelihood estimator is asymptotically equal to $1/(n I(p))$ as the length of a sequence $n$ approaches infinity.  Thus, the fact that the Fisher information is approximately $790$ for $p = 0.5$ and $20,000$ for $p = 0.1$ implies that given any target accuracy, it takes about $25$ times as much data to predict a weighting near $p = 0.5$ as it does to predict a weighting near $p = 0.1$ with the same accuracy.
%
%
%

%
%Probably should mention the Cramer-Rao bound.
%
%Fisher information among independent RVs is \emph{additive}.  Thus, we know how much more data we'd need to get the same level of information about $\theta$ depending on its value!

\section{Acknowledgements}
This work was supported by funds from the National Institutes of Health (R01GM126554 to CH) and the National Science Foundation (DMS1344199 to CH).

\newpage
\bibliographystyle{unsrt}

\bibliography{DeconvolvingBib.bib}

\begin{thebibliography}{10}

\bibitem{Leonard:2013}
Christopher~W. Leonard, Christine~E. Hajdin, Fethullah Karabiber, David~H.
  Mathews, Oleg~V. Favorov, Nikolay~V. Dokholyan, and Kevin~M. Weeks.
\newblock Principles for understanding the accuracy of {SHAPE}-directed {RNA}
  structure modeling.
\newblock {\em Biochemistry}, 52(4):588--595, 2013.
\newblock PMID: 23316814.

\bibitem{Spasic:2018}
Aleksandar Spasic, Sarah~M Assmann, Philip~C Bevilacqua, and David~H Mathews.
\newblock Modeling {RNA} secondary structure folding ensembles using {SHAPE}
  mapping data.
\newblock {\em Nucleic Acids Res}, 46(1):314--323, Jan 2018.

\bibitem{Lu:2016}
Zhipeng Lu, Qiangfeng~Cliff Zhang, Byron Lee, Ryan~A. Flynn, Martin~A. Smith,
  James~T. Robinson, Chen Davidovich, Anne~R. Gooding, Karen~J. Goodrich,
  John~S. Mattick, and et~al.
\newblock {RNA} duplex map in living cells reveals higher-order transcriptome
  structure.
\newblock {\em Cell}, 165(5):1267--1279, May 2016.

\bibitem{Antunes:2018}
Deborah Antunes, Natasha A.~N. Jorge, Ernesto~R. Caffarena, and Fabio Passetti.
\newblock Using {RNA} sequence and structure for the prediction of riboswitch
  aptamer: A comprehensive review of available software and tools.
\newblock {\em Frontiers in Genetics}, 8, Jan 2018.

\bibitem{Schroeder:2018}
Susan~J. Schroeder.
\newblock Challenges and approaches to predicting {RNA} with multiple
  functional structures.
\newblock {\em RNA}, 24(12):1615--1624, Aug 2018.

\bibitem{mathews-turner-06}
David~H Mathews and Douglas~H Turner.
\newblock Prediction of {RNA} secondary structure by free energy minimization.
\newblock {\em Curr Opin Struct Biol}, 16(3):270--278, 2006.

\bibitem{turner-mathews-10}
Douglas~H Turner and David~H Mathews.
\newblock {NNDB}: the nearest neighbor parameter database for predicting
  stability of nucleic acid secondary structure.
\newblock {\em Nucleic Acids Res}, 38(suppl 1):D280--D282, 2010.

\bibitem{Zuker:1981}
M~Zuker and P~Stiegler.
\newblock Optimal computer folding of large {RNA} sequences using
  thermodynamics and auxiliary information.
\newblock {\em Nucleic Acids Research}, 9(1):133--148, 01 1981.

\bibitem{zuker-89}
Michael Zuker.
\newblock On finding all suboptimal foldings of an {RNA} molecule.
\newblock {\em Science}, 244(4900):48--52, 1989.

\bibitem{wuchty-etal-99}
Stefan Wuchty, Walter Fontana, Ivo~L Hofacker, and Peter Schuster.
\newblock Complete suboptimal folding of {RNA} and the stability of secondary
  structures.
\newblock {\em Biopolymers}, 49(2):145--165, 1999.

\bibitem{mccaskill-90}
John~S McCaskill.
\newblock The equilibrium partition function and base pair binding
  probabilities for {RNA} secondary structure.
\newblock {\em Biopolymers}, 29(6-7):1105--1119, 1990.

\bibitem{hofacker-etal-94}
Ivo~L Hofacker, Walter Fontana, Peter~F Stadler, L~Sebastian Bonhoeffer,
  Manfred Tacker, and Peter Schuster.
\newblock Fast folding and comparison of {RNA} secondary structures.
\newblock {\em Monatshefte f{\"u}r Chemie/Chemical Monthly}, 125(2):167--188,
  1994.

\bibitem{mathews-04}
David~H Mathews.
\newblock Using an {RNA} secondary structure partition function to determine
  confidence in base pairs predicted by free energy minimization.
\newblock {\em RNA}, 10(8):1178--1190, 2004.

\bibitem{Ding:2003}
Ye~Ding and Charles~E Lawrence.
\newblock A statistical sampling algorithm for {RNA} secondary structure
  prediction.
\newblock {\em Nucleic Acids Research}, 31(24):7280--7301, 12 2003.

\bibitem{Deigan:2009}
Katherine~E. Deigan, Tian~W. Li, David~H. Mathews, and Kevin~M. Weeks.
\newblock Accurate {SHAPE}-directed {RNA} structure determination.
\newblock {\em Proceedings of the National Academy of Sciences},
  106(1):97--102, 2009.

\bibitem{Zarringhalam:2012}
Kourosh Zarringhalam, Michelle~M. Meyer, Ivan Dotu, Jeffrey~H. Chuang, and
  Peter Clote.
\newblock Integrating chemical footprinting data into {RNA} secondary structure
  prediction.
\newblock {\em PLOS ONE}, 7(10):1--13, 10 2012.

\bibitem{Washietl:2012}
Stefan Washietl, Ivo~L Hofacker, Peter~F Stadler, and Manolis Kellis.
\newblock {RNA} folding with soft constraints: reconciliation of probing data
  and thermodynamic secondary structure prediction.
\newblock {\em Nucleic Acids Res}, 40(10):4261--4272, May 2012.

\bibitem{Eddy:2014}
Sean~R. Eddy.
\newblock Computational analysis of conserved {RNA} secondary structure in
  transcriptomes and genomes.
\newblock {\em Annual Review of Biophysics}, 43(1):433--456, 2014.
\newblock PMID: 24895857.

\bibitem{Sukosd:2013}
Zsuzsanna S{\"u}k{\"o}sd, M~Shel Swenson, J{\o}rgen Kjems, and Christine~E
  Heitsch.
\newblock Evaluating the accuracy of {SHAPE}-directed {RNA} secondary structure
  predictions.
\newblock {\em Nucleic Acids Research}, 41(5):2807--2816, 03 2013.

\bibitem{Hobartner:2003}
Claudia H{\"o}bartner and Ronald Micura.
\newblock Bistable secondary structures of small {RNA}s and their structural
  probing by comparative imino proton {NMR} spectroscopy.
\newblock {\em Journal of Molecular Biology}, 325(3):421 -- 431, 2003.

\bibitem{Nussinov:1980}
R~Nussinov and A~B Jacobson.
\newblock Fast algorithm for predicting the secondary structure of
  single-stranded {RNA}.
\newblock {\em Proceedings of the National Academy of Sciences of the United
  States of America}, 77(11):6309--6313, 11 1980.

\bibitem{Clote:2005}
P.~Clote.
\newblock An efficient algorithm to compute the landscape of locally optimal
  {RNA} secondary structures with respect to the {N}ussinov--{J}acobson energy
  model.
\newblock {\em Journal of Computational Biology}, 12(1):83--101, 2017/12/13
  2005.

\bibitem{Clote:2007}
Peter Clote, Evangelos Kranakis, Danny Krizanc, and Ladislav Stacho.
\newblock Asymptotic expected number of base pairs in optimal secondary
  structure for random {RNA} using the {N}ussinov--{J}acobson energy model.
\newblock {\em Discrete Applied Mathematics}, 155(6):759 -- 787, 2007.
\newblock Computational Molecular Biology Series, Issue V.

\bibitem{Vieweger:2018}
Mario Vieweger and David~J. Nesbitt.
\newblock Synergistic {SHAPE}/single-molecule deconvolution of {RNA}
  conformation under physiological conditions.
\newblock {\em Biophysical Journal}, 114(8):1762--1775, Apr 2018.

\bibitem{Quarrier:2010}
Scott Quarrier, Joshua~S Martin, Lauren Davis-Neulander, Arthur Beauregard, and
  Alain Laederach.
\newblock Evaluation of the information content of {RNA} structure mapping data
  for secondary structure prediction.
\newblock {\em RNA}, 16(6):1108--1117, 06 2010.

\bibitem{Rice:2014}
Greggory~M Rice, Christopher~W Leonard, and Kevin~M Weeks.
\newblock {RNA} secondary structure modeling at consistent high accuracy using
  differential shape.
\newblock {\em RNA}, 20(6):846--854, 06 2014.

\bibitem{Lorenz:2011}
Ronny Lorenz, Stephan~H Bernhart, Christian H{\"o}ner~zu Siederdissen, Hakim
  Tafer, Christoph Flamm, Peter~F Stadler, and Ivo~L Hofacker.
\newblock {ViennaRNA} package 2.0.
\newblock {\em Algorithms for Molecular Biology : AMB}, 6:26--26, 2011.

\bibitem{Cordero:2015}
Pablo Cordero and Rhiju Das.
\newblock Rich {RNA} structure landscapes revealed by mutate-and-map analysis.
\newblock {\em {PLOS} Computational Biology}, 11(11):e1004473, Nov 2015.

\bibitem{Woods:2017}
Chanin~T. Woods, Lela Lackey, Benfeard Williams, Nikolay~V. Dokholyan, David
  Gotz, and Alain Laederach.
\newblock Comparative visualization of the {RNA} suboptimal conformational
  ensemble in vivo.
\newblock {\em Biophysical Journal}, 113(2):290 -- 301, 2017.

\bibitem{Durrett:2019}
Rick Durrett.
\newblock {\em Probability: Theory and Examples}.
\newblock Cambridge University Press, 5th edition, 2019.

\bibitem{Rogers:2014}
Emily Rogers and Christine~E Heitsch.
\newblock Profiling small {RNA} reveals multimodal substructural signals in a
  {B}oltzmann ensemble.
\newblock {\em Nucleic Acids Research}, 42(22):e171--e171, 12 2014.

\bibitem{Wasserman:2004}
Larry Wasserman.
\newblock {\em All of Statistics: A Concise Course in Statistical Inference}.
\newblock Springer, 1st edition, 2004.

\end{thebibliography}

\end{document}